\documentclass[11pt]{article}  
\usepackage{amsmath}
\usepackage{amssymb}
\usepackage{theorem}
\usepackage{euscript}
\usepackage{pstricks}
\usepackage{authblk}
\usepackage{verbatim}
\usepackage{graphics}
\usepackage{graphicx}
\usepackage{tikz}
\usepackage{amscd}
\usepackage{color}

\usepackage{hyperref}
\hypersetup
{
	colorlinks,
	citecolor=blue,
	filecolor=green,
	linkcolor=green,
	urlcolor=blue
}
\hypersetup{linktocpage}

\newtheorem{theorem}{Theorem}[section]
\newtheorem{lemma}[theorem]{Lemma}

\numberwithin{equation}{section}

\theoremstyle{definition}
\newtheorem{definition}[theorem]{Definition} 
\theoremstyle{remark}

\newcommand{\brac}[1]{\left(#1\right)}

\newcommand{\bk}{{\boldsymbol{k}}}
\newcommand{\bell}{{\boldsymbol{\ell}}}

\newcommand{\bx}{{\boldsymbol{x}}}

\newcommand{\bX}{{\boldsymbol{X}}}
\newcommand{\by}{{\boldsymbol{y}}}

\newcommand{\bW}{{\boldsymbol{W}}}

\newcommand{\balpha}{{\boldsymbol{\alpha}}}

\newcommand{\rd}{{\rm d}} 
\newcommand{\Int}{{\rm Int}}

\def\II{\mathbb I}

\def\ZZd{{\mathbb Z}^d}
\def\IId{{\mathbb I}^d}
\def\ZZ{{\mathbb Z}}

\def\RR{{\mathbb R}}
\def\RRd{{\mathbb R}^d}

\def\NN{{\mathbb N}}
\def\NNd{{\NN}^d}

\def\II{{\mathbb I}}

\def\NN{{\mathbb N}}
\def\RR{{\mathbb R}}

\def\UU{{\mathbb U}}
\def\UUd{{\mathbb U}^d}

\def\IId{{\mathbb I}^d}

\def\NNd{{\mathbb N}^d}
\def\RRd{{\mathbb R}^d}

\def\ZZd{{\mathbb Z}^d}

\def\Hh{{\mathcal H}}

\def\II{{\mathbb I}}

\def\ZZ{{\mathbb Z}}
\def\NN{{\mathbb N}}
\def\RR{{\mathbb R}}

\def\IId{{\mathbb I}^d}

\def\NNd{{\mathbb N}^d}
\def\RRd{{\mathbb R}^d}

\def\Wpmix{W^s_p(\IId)}

\def\Wpgamma{W^s_p(\RRd,\gamma)}
\def\Wap{W^s_p}
\def\Wa{W^s_2(\RRd,\gamma)}

\def\BWpgamma{\bW^s_p(\RRd,\gamma)}

\newcommand{\norm}[2]{\left\|{#1}\right\|_{#2}}

\title{\sffamily {Optimal numerical integration  for functions in fractional Gaussian Sobolev spaces}}

\author{Van Kien Nguyen}
\affil{Department of Mathematical Analysis, University of Transport and Communications
	\protect\\	No.3 Cau Giay Street, Lang  Ward,
	Hanoi, Vietnam
	\protect\\
	Email: kiennv@utc.edu.vn}

\date{\today}
\begin{document}
\maketitle

\begin{abstract}
 This paper investigates the numerical approximation of integrals for functions in fractional Gaussian Sobolev spaces $W^s_{p}(\mathbb{R}^d,\gamma)$ with dominating mixed smoothness defined via kernel related to the fractional Ornstein-Uhlenbeck operator. Building upon quadrature rules for fractional Sobolev spaces on the unit cube $[-\tfrac{1}{2}, \tfrac{1}{2}]^d$, we construct quadrature schemes on $\mathbb{R}^d$ that achieve the same rate of convergence. As a consequence, we establish the optimal asymptotic order of the integration error in the regime $1 < p < \infty$ and $s > \frac{1}{p}$, $s\not \in \mathbb{N}$.
 Furthermore, we show that the fractional Gaussian Sobolev spaces $W^s_{2}(\mathbb{R}^d,\gamma)$ coincide with Hermite spaces $\mathcal{H}^s(\mathbb{R}^d,\gamma)$ characterized by the weighted $\ell_2$-summability of their Fourier-Hermite coefficients. From this, we derive the optimal asymptotic order of the integration error for functions in these spaces for all $s > \frac{1}{2}$. We also establish the corresponding optimal asymptotic order for functions in fractional Gaussian Sobolev spaces $W^s_{p,G}(\mathbb{R}^d,\gamma)$ defined via the Gagliardo seminorm.
	
	
	
\end{abstract}

\section{Introduction}
 \label{Introduction}
Let $\rd\gamma (\bx) = \rho(\bx) \rd\bx$  be the $d$-dimensional standard Gaussian measure on $\RRd$  with the density 
$$
\rho(\bx):=(2\pi)^{-d/2} \exp\brac{-|\bx|^2/2},\ \ \bx\in \RRd.
$$
In this paper we investigate numerical approximation of the integrals 
\begin{equation} \label{I(f)}
	I^\gamma(f):=\int_{\RRd} f(\bx) \, \rd \gamma(\bx) = \int_{\RRd} f(\bx) \rho(\bx) \, \rd\bx
\end{equation}
for functions $f$ belonging  to the fractional Gaussian Sobolev spaces  with dominating mixed smoothness
$W^s_{p}(\RRd,\gamma)$ or $W^s_{p,G}(\mathbb{R}^d, \gamma)$, where $1 < p < \infty$, $s>\frac{1}{p}$, $s\not \in\NN$  (see Section \ref{Numerical integration} for the definition). We denote these spaces by $W^s_{p,\gamma}$. 
To approximate this integral, we employ a (linear) quadrature scheme of the form
\begin{equation} \label{I_n(f)-introduction}
	I_n(f) := \sum_{i=1}^n \lambda_{i,n}  f(\bx_{i,n}),\ \ n\in \NN,
\end{equation}
with  the convention $I_0(f) = 0$, where $\{\bx_{1,n},\ldots,\bx_{n,n}\}\subset \RRd$  are given integration nodes and $\{\lambda_{1,n},\ldots,\lambda_{n,n}\}$ are the corresponding weights. For convenience, we permit the integration nodes to coincide.
Let $\boldsymbol{W}^s_{p,\gamma}$ be the unit ball of $W^s_{p,\gamma}$. The optimality  of  quadratures  for  
$W^s_{p,\gamma}$ is measured by the quantity
\begin{equation*} \label{Int_n}
\Int_n(\boldsymbol{W}^s_{p,\gamma}) :=\inf_{I_n}\sup_{f\in \boldsymbol{W}^s_{p,\gamma}}|I^\gamma(f)-I_n(f)|,
\end{equation*}
where the infimum is taken over all quadratures of the form \eqref{I_n(f)-introduction}.   

We are interested in the asymptotic order of this quantity as 
$n\to \infty$, as well as in the construction of asymptotically optimal quadrature rules. In this paper, we do not address the dependence on the dimension or issues of tractability.
The problem of multivariate numerical integration \eqref{I(f)}--\eqref{I_n(f)-introduction} has been studied in \cite{IKLP2015, IL2015, DILP18} for functions in certain Hermite spaces. In particular, the space 
$\Hh_{d,s}$,  $s\in \NN$,  considered in \cite{DILP18} coincides with 
$W^s_2(\mathbb{R}^d, \gamma)$,  in the sense of  equivalent norms.
For $s\in \NN$, it was shown in \cite{DILP18} that 
	\begin{equation*}\label{DILP18}
n^{-s} (\log n)^{\frac{d-1}{2}} 
\ll	
\Int_n\big( \bW^s_2(\RRd, \gamma)\big)  
\ll 
n^{-s} (\log n)^{\frac{d(2s + 3)}{4} - \frac{1}{2}}.
\end{equation*}
The lower bound is obtained by constructing a fooling function in 
$\in \boldsymbol{W}^s_{2}(\RRd,\gamma)$
whose support does not contain any of the integration nodes.
The upper bound is achieved by a translated and scaled quasi-Monte Carlo quadrature rule, based on Dick’s higher order digital nets, mapping the unit cube $[0,1]^d$ to  $[-b,b]^d$ with $b=2\sqrt{s\log n}$. 

The above result was subsequently improved by Dinh D\~ung and the author in \cite{DN23} for $s\in \NN$ and $1<p<\infty$
\begin{equation}\label{DK}
\Int_n\big( \bW^s_p(\RRd, \gamma)\big)  
 \asymp 	n^{-s} (\log n)^{\frac{d-1}{2}} .
\end{equation}
To derive the upper bound, we propose in \cite{DN23} a novel approach in which an asymptotically optimal quadrature rule on the unit $d$-cube is extended to its integer translates and subsequently assembled into a global quadrature rule on $\mathbb{R}^d$. As a consequence, we show that the asymptotic integration error for Gaussian Sobolev spaces is of the same order as that for Sobolev spaces on the unit cube, as given in \eqref{DK}.

The aim of this paper is to extend the result in \eqref{DK} to fractional Gaussian Sobolev spaces with dominating mixed smoothness $W^s_{p,\gamma}$ in the regime $1 < p < \infty$, $s > \frac{1}{p}$, and $s \notin \mathbb{N}$. More precisely, under these conditions, we construct an asymptotically optimal quadrature rule $I_n$ of the form \eqref{I_n(f)-introduction}, which achieves the asymptotically optimal order of convergence
\begin{equation*} 	\label{AsympQuadrature}
\Int_n\big(\boldsymbol{W}^s_{p}(\RRd,\gamma)\big) 
	\asymp
	n^{-s} (\log n)^{(d-1)\big(1-\frac{1}{p}\big)}.
\end{equation*}

In this paper, we further show that the Gaussian Sobolev spaces $W^s_{2}(\mathbb{R}^d,\gamma)$ can be identified with Hermite spaces $\mathcal{H}^s(\RRd,\gamma)$ defined via the weighted $\ell_2$-summability of their Fourier-Hermite coefficients. It then follows that for all
$s>\frac{1}{2}$ we have
\begin{equation*} 
	 	\Int_n\big(\boldsymbol{\mathcal{H}}^s(\RRd,\gamma)\big)\asymp n^{-s} (\log n)^{\frac{d-1}{2}} . 
\end{equation*}
Here $\boldsymbol{\mathcal{H}}^s(\RRd,\gamma)$ is the unit ball of  $\mathcal{H}^s(\mathbb{R}^d,\gamma)$. This result provides a complete characterization of the asymptotic convergence rate of the integration error for functions in the Hermite space $\mathcal{H}^s(\mathbb{R}^d,\gamma)$.

Moreover, this paper establishes the optimal convergence rate of the integration error for functions in fractional Gaussian Sobolev spaces $W^s_{p,G}(\mathbb{R}^d,\gamma)$ defined via the Gagliardo seminorm in the regime $2 < p < \infty$, $s > \frac{1}{p}$, and $s \notin \mathbb{N}$, namely
\begin{equation*}
	\Int_n\big(\boldsymbol{W}^s_{p,G}(\mathbb{R}^d,\gamma)\big)
	\asymp
	n^{-s} (\log n)^{(d-1)\left(1-\frac{1}{p}\right)}.
\end{equation*}

The paper is organized as follows. In Section~\ref{Numerical integration}, we introduce fractional Gaussian Sobolev spaces, discuss their main properties, and prove their relationship with Hermite spaces. Section~\ref{sec:numer} is devoted to the construction of asymptotically optimal quadrature rules for fractional Gaussian Sobolev spaces. In Section~\ref{Sec-4}, we provide numerical experiments supporting the results obtained in Section \ref{sec:numer}.
%%%%%%%%%%%%%%%%%%%%%%%%%%%%%%%%%%

\noindent
{\bf Notation.}  
The letter $d$ is always reserved for
the underlying dimension of $\RR^d$, $\NN^d$, etc. Vectors in $\RRd$  are denoted by boldface
letters. For $\bx \in \RR^d$, $x_i$ denotes the $i$th coordinate, i.e., $\bx := (x_1,\ldots, x_d)$.  If $ 1\le p< \infty$, we write
$|\bx|_p := \big(\sum_{i=1}^d |x_i|^p\big)^{1/p}$ and $|\bx|_\infty=\max\{|x_1|,\ldots,|x_d|\}$. When $p=2$ we simply write $|\bx|$. 
For the quantities $A_n$ and $B_n$ depending on 
$n$ in an index set $J$  
we write  $A_n \ll B_n$  
if there exists some constant $C >0$ independent of $n$ such that 
$A_n \leq CB_n$ for all $n \in J$, and  
$A_n \asymp B_n$ if $A_n  \ll B_n $
and $B_n  \ll A_n $. General positive constants or positive constants depending on parameters $s, d,\ldots$ are denoted by $C$ or $C_{s,d,\ldots}$, respectively. Values of constants $C$ and  $C_{s,d}$ may be different in various places.  Denote by $|G|$ the cardinality of the finite set $G$. The unit ball in Banach space $X$ is denoted by the bold symbol $\bX$.

%%%%%%%%%%%%%%%%%%%%%%%%%%%%%%%%%
%%%%%%%%%%%%%%%%%%%%%%%%%%%%%%%%%

\section{Fractional Gaussian Sobolev spaces}
\label{Numerical integration}

\label{Subsec-AssemblingQuadratures}
In this section, we first introduce fractional Gaussian Sobolev spaces with dominating mixed smoothness, extending Sobolev–Slobodeckij spaces of mixed smoothness to the Gaussian setting. We then define these spaces via a kernel associated with the fractional Ornstein–Uhlenbeck operator and discuss their relationship with Hermite spaces.

Let  $1\leq p<\infty$.  
	We define the Gaussian Lebesgue space  $L_p(\RRd,\gamma)$ to be the set of all measurable functions $f$ on $\RRd$ such that the norm
	$$
	\|f\|_{L_p(\RRd,\gamma)} : = \bigg( \int_{\RRd} |f(\bx)|^p \rd \gamma(\bx)\bigg)^{1/p}
	=
	\bigg( \int_{\RRd} |f(\bx)|^p \rho(\bx) \rd \bx\bigg)^{1/p} \ <  \ \infty. 
	$$
	For $s \in \NN$, we define the Gaussian Sobolev  space with mixed smoothness $\Wap(\RRd,\gamma)$ as the normed space of all functions $f\in L_p(\RRd,\gamma)$ such that the  generalized  partial derivatives $D^\balpha f$ of order $\balpha$  belong to $L_p(\RRd,\gamma)$ for all $\balpha\in \NN_0^d$ satisfying $|\balpha|_\infty\leq s$. The norm of a  function $f$ in this space 
	is defined by
	\begin{align} \label{W-Omega}
		\|f\|_{\Wap(\RRd,\gamma)}: = \Bigg(\sum_{|\balpha|_\infty \leq s} \|D^\balpha f\|_{L_p(\RRd,\gamma)}^p\Bigg)^{1/p}.
	\end{align}

Before defining fractional Gaussian Sobolev spaces with dominating mixed smoothness on $\RRd$, we introduce the corresponding fractional Sobolev spaces on $\UU^d$, where $\UU=[a,b]$ or $\UU=\RR$.
As usual,   $L_p(\UUd)$ denotes the Lebesgue space of functions on $\UUd$ equipped with the standard $L_p$-norm. For $s\in \NN$, the  space $\Wap(\UUd)$ is defined analogously to $\Wap(\RRd,\gamma)$ by replacing $L_p(\RRd,\gamma)$ in \eqref{W-Omega} with $L_p(\UUd)$. 

For function $f$ defined on a set $\UUd$ and $\bx,\by\in \UUd$ we denote
$$
\Delta_{y_j,j} f(\bx):=f(x_1,\ldots,x_{j-1},x_j,x_{j+1},\ldots,x_d)-f(x_1,\ldots,x_{j-1},y_j,x_{j+1},\ldots,x_d) .
$$ 
If $f$ is a univariate function and $x,y\in \UU$, we simply write $\Delta_yf(x)$. 
For $e\subset [d]:= \{1,\ldots,d\}$, $e\not=\emptyset$,  we denote   $\by_e=(y_j)_{j\in e}\in \RR^{|e|}$ and
$$
\Delta^e_{\by_e}f(\bx):=\prod_{j\in e} \Delta_{y_j,j} f(\bx), 
$$
where $\Delta_{y_j,j}$ is the univariate operator applied to the $j$-th coordinate of $f$
with the other variables kept fixed. 

For $s>0$ and $s \notin \mathbb{N}$, we write $s = \bar{s} + \tilde{s}$, where $\bar{s} := \lfloor s \rfloor$ is the integer part of $s$ and $\tilde{s}$ is the fractional part. 
\begin{definition}\label{def:Sobolev}
	Let $1\leq p<\infty$, $s>0$ and $s\not \in \NN$. Then the fractional Sobolev space with mixed smoothness $\Wap(\UUd)$ is defined by
	$$
	\Wap(\UUd):=\big\{f\in W_p^{\bar{s}}(\UUd): [f]_{\Wap(\UUd)} <\infty \big\}
	,$$
	where
	$$
	[f]_{\Wap(\UUd)}
	:=\sum_{|\balpha|_\infty = \bar{s}}\sum_{e\subset [d],e\not=\emptyset}
	\Bigg(
	\int_{\UU^{|e|}}\int_{\UUd}
	\frac{
		\big| \Delta^e_{\by_e}D^\balpha f(\bx)  \big|^p
	}{
		\prod_{j\in e}	|x_j - y_j|^{1 + \tilde{s}p}
	}
	\, \rd\bx \, \rd\by_e
	\Bigg)^{\frac{1}{p}}.
	$$
	The norm of $f\in \Wap(\UUd)$ is given by
	$$
	\|f \|_{\Wap(\UUd)}:
	=
	\|f \|_{ W_p^{\bar{s}}(\UUd)}
	+ [f]_{\Wap(\UUd)}.
	$$
\end{definition}
Note that the space $\Wap(\UUd)$ coincides with the so-called Besov space of dominating mixed smoothness, often denoted by $S^s_{p,p}B(\UUd)$, see \cite[Theorems 2.2.6.2 and 2.3.4.1]{ST87B}.    

\begin{definition}\label{def:Sobolev-gamma}
Let $1\leq p<\infty$,  $s>0$ and $s\not \in \NN$. Then the fractional Gaussian Sobolev space with mixed smoothness $W^s_{p,G}(\RRd,\gamma)$ is defined by
	$$
W^s_{p,G}(\RRd,\gamma):=\big\{f\in W_{p}^{\bar{s}}(\RRd,\gamma): [f]_{W^s_{p,G}(\RRd,\gamma)} <\infty \big\}
	,$$
where
\begin{equation*}\label{eq:quasi-norm}
[f]_{W^s_{p,G}(\RRd,\gamma)}
:=
\sum_{|\balpha|_\infty = \bar{s}}\sum_{e\subset [d],e\not=\emptyset}
\Bigg(
\int_{\RR^{|e|}}\int_{\RRd}
\frac{
	\big|\Delta^e_{\by_e}D^\balpha f(\bx)    \big|^p
}{
	\prod_{j\in e}	|x_j - y_j|^{1 + \tilde{s}p}
}
\, \rd\gamma(\bx) \, \rd\gamma(\by_e)
\Bigg)^{\frac{1}{p}}.	 
\end{equation*}
The norm of $f\in W^s_{p,G}(\RRd,\gamma)$ is given by
$$
\|f \|_{W^s_{p,G}(\RRd,\gamma)}:
=
\|f \|_{ W_p^{\bar{s}}(\RRd,\gamma)}
+ [f]_{W^s_{p,G}(\RRd,\gamma)}.
$$
\end{definition}

\begin{lemma}\label{lem:continuous}
Let  $1\leq p<\infty$ and $s>\frac{1}{p}$,  $s\not \in \NN$. Then every $f\in W^s_{p,G}(\RRd,\gamma)$ is continuous on $\RRd$.
\end{lemma}  
\begin{proof} Let $f\in W^s_{p,G}(\RRd,\gamma)$. For any $\ell>0$, consider the restriction of $f$ to the cube $[-\ell,\ell]^d\subset \RRd$, denoted by $f_\ell$. Then by Definitions \ref{def:Sobolev} and \ref{def:Sobolev-gamma} we get $f_\ell\in \Wap([-\ell,\ell]^d)$. It is well known that $\Wap([-\ell,\ell]^d)$ is continuously embedded into the space of bounded continuous functions on $[-\ell,\ell]^d$, see, e.g., \cite{HV09}. Hence $f_\ell$ is   continuous on $[-\ell,\ell]^d$. Since $\ell>0$ is arbitrary it follows that 
$f$ is   continuous on $\RRd$.
	\hfill
\end{proof}

%%%%%%%%%%%%%%%%%%%%%%%%

Next, we introduce fractional Gaussian Sobolev spaces defined via kernel 
related to the fractional Ornstein-Uhlenbeck operator. 
In order to do so, we introduce the Ornstein--Uhlenbeck semigroup and its generator $L_\gamma$. Let $t > 0$. For $v \in L_1(\mathbb{R},\gamma)$  the Ornstein--Uhlenbeck semigroup is defined as
	\[
	e^{tL_\gamma} v(x)
	:=
	\int_{\mathbb{R}} M_t(x,y)\, v(y)\, \rd\gamma(y),\ \ x \in \mathbb{R},
	\]
	where $M_t(x,y)$ is the Mehler kernel
	\[
	M_t(x,y)
	:=
	\frac{1}{(1 - e^{-2t})^{1/2}}
	\exp\left(
	- \frac{e^{-2t}|x|^2 - 2e^{-t} x  y + e^{-2t}|y|^2}{2(1 - e^{-2t})}
	\right).
	\]
Ornstein--Uhlenbeck operator in $\RR$ is a second-order differential operator
given by
	\[
	L_\gamma v(x) = \frac{\rd^2}{\rd x^2} v(x) - x\frac{\rd }{\rd x}v(x).
	\]
	For $k\in \NN_0$, the normalized probabilistic Hermite polynomial
	$H_k$ of degree $k$ on $\RR$ is defined by
	\begin{equation}\label{eq:hermite} 
		H_k(x) 
		:= 
		\frac{(-1)^k}{\sqrt{k!}} 
		\exp\left(\frac{x^2}{2}\right) \frac{\rd^k}{\rd x^k} \exp\left(-\frac{x^2}{2}\right) .
	\end{equation}
Then the Hermite polynomial $H_k$
are eigenfunctions of $L_\gamma $  with eigenvalues $k$, i.e.,
$$
L_\gamma H_k(x)=kH_k(x),\ \ \ k\in \NN_0.
$$
We refer the reader to \cite{LMP20} and the references therein for the main properties of $e^{tL_\gamma}$ and $L_\gamma$.
	
For $\sigma \in (0,1)$, we define the fractional Ornstein--Uhlenbeck operator by means of spectral decomposition via the Bochner subordination formula, i.e.,
\begin{equation}\label{eq:OU-operator}
\begin{aligned}
	(-L_\gamma)^\sigma v(x)
	&:= \frac{1}{\Gamma(-\sigma)} \int_0^\infty \frac{e^{tL_\gamma}v(x) - v(x)}{t^{\sigma+1}}\, \rd t \\
	&= \frac{1}{\Gamma(-\sigma)} \int_0^\infty \frac{\rd t}{t^{\sigma+1}} \int_{\mathbb{R}} M_t(x,y)\,(v(y) - v(x))\, \rd\gamma(y) \\
	&= \frac{1}{\Gamma(-\sigma)} \int_{\mathbb{R}} (v(x) - v(y))\, K_{2\sigma}(x,y)\, \rd\gamma(y),
\end{aligned}	 
\end{equation}
where \begin{equation*}
	K_\sigma(x,y) := \int_0^\infty \frac{M_t(x,y)}{t^{\frac{\sigma}{2}+1}}\, \rd t \,.
\end{equation*}
Note that the right-hand side in \eqref{eq:OU-operator} has to be understood in the Cauchy principal value sense. We have that
\begin{equation}\label{eq:eigen}
(-L_\gamma)^\sigma H_k(x)=k^\sigma H_k(x),\ \ k\in\NN,	 
\end{equation}
i.e., the Hermite polynomials $H_k$ are eigenfunctions of $(-L_\gamma)^\sigma$, with corresponding eigenvalues $k^\sigma$. 
\begin{definition}\label{def:Kfunction}
	Let $1\leq p<\infty$,  $s >0$ and $s\not \in \NN$. We define the fractional Gaussian Sobolev space  with mixed smoothness $W_{p}^s(\mathbb{R}^d,\gamma)$ as
	\[
	W_{p}^s(\mathbb{R}^d,\gamma)
	:=
	\left\{
	f \in W^{\bar{s}}_p(\mathbb{R}^d,\gamma)
	\; ; \;
	[ f]_{W_{p}^s(\RRd,\gamma)} < \infty
	\right\},
	\]
	where
\begin{equation}\label{eq:quasi-norm2}
	[ f]_{W_{p}^s(\mathbb{R}^d,\gamma)}
:=
\sum_{|\balpha|_\infty = \bar{s}}\sum_{e\subset [d],e\not=\emptyset}
\Bigg(\int_{\RR^{|e|}}
\int_{\RRd}	
\big| \Delta^e_{\by_e}D^\balpha f(\bx)  \big|^p \bigg(\prod_{i\in e}K_{p\tilde{s}}(x_i,y_i)\bigg)
\rd \gamma(\bx) \, \rd\gamma(\by_e)
\Bigg)^{\frac{1}{p}}.	 
\end{equation} 
The norm of $f\in W_{p}^s(\mathbb{R}^d,\gamma)$ is defined by
$$
\|f\|_{W_{p}^s(\mathbb{R}^d,\gamma)}:=\|f\|_{W^{\bar{s}}_p(\mathbb{R}^d,\gamma)}+	[ f]_{W_{p}^s(\mathbb{R}^d,\gamma)}.
$$
\end{definition}
In the following we collect some properties of the kernel $K_\sigma(x,y)$ which is useful to formulate our results. First, the following lemma was proved in \cite[Lemma 2.8]{CCMP22}.
\begin{lemma}\label{lem:help}For $x,y\in \RR$ and $\sigma>0$ we have
	$$ \frac{2^{\sigma+\frac{1}{2}}\Gamma(\frac{\sigma+1}{2})}{|x-y|^{1+\sigma}} \leq 
	K_\sigma(x,y) \,.
	$$
\end{lemma}
We have some further estimates.
\begin{lemma}  \label{lem:estimate}
	Let $\sigma>0$.
	\begin{enumerate}
		\item[(1)] If $x,y\in [-\ell,\ell]$, $\ell>0$, then it holds
		$$
		K_\sigma(x,y) \le \frac{C_\ell}{|x-y|^{1+\sigma}}.
		$$
		\item[(2)] If $|x-y|\leq 1$ and $t_0>0$, then
		$$
		K_\sigma(x,y)\geq C_{t_0} 	\exp\bigg(\frac{x^2+y^2}{4} \Big(1-\frac{ 1-e^{-t_0}}{ 1+e^{-t_0}}\Big)\bigg) \frac{1}{|x-y|^{1+\sigma}},
		$$
where $C_{t_0}$ is a positive constant depending only on $t_0$.
	\end{enumerate}
	
\end{lemma}
\begin{proof}
	By the definition of Mehler kernel we have
	\begin{equation}\label{eq:Mehler}
		\begin{aligned}
			M_t(x,y)
			&=	\exp\bigg(\frac{x^2+y^2}{4}-\frac{(x^2+y^2)(1-e^{-t})}{4(1+e^{-t})}\bigg)\frac{1}{(1 - e^{-2t})^{1/2}}
			\exp\left(
			- \frac{e^{-t}(x-y)^2}{2(1 - e^{-2t})}
			\right).
		\end{aligned}
	\end{equation}
	Observe that
	\[
	\frac{1}{(1 - e^{-2t})^{1/2}}
	\exp\left(
	- \frac{e^{-t}(x-y)^2}{2(1 - e^{-2t})}
	\right)
	\sim \frac{1}{(2 t)^{1/2}} \exp\left(
	- \frac{(x-y)^2}{4t}
	\right), \ \ t\to 0.
	\]
	From this and
	$$
\frac{1}{(1 - e^{-2t})^{1/2}}
\exp\left(
- \frac{e^{-t}(x-y)^2}{2(1 - e^{-2t})}
\right) \leq C
	$$
when $t\geq 1$
	we get for $x,y\in [-\ell,\ell]$
	\begin{equation*}
		\begin{aligned}
			{K}_\sigma(x,y) &\leq C \exp\Big(\frac{x^2+y^2}{4}\Big)\bigg(\int_{1}^{\infty}  
			\frac{\rd t}{t^{\frac{\sigma}{2}+1}}, + \int_0^{\infty}\frac{\exp\big(
				- \frac{(x-y)^2}{4t}
				\big)}{(2 t)^{1/2}t^{\frac{\sigma}{2}+1}}\rd t
		\bigg)
		\\
		&\leq C_\ell\bigg(1+\int_0^{\infty}\frac{\exp\big(
			- \frac{(x-y)^2}{4t}
			\big)}{ t^{\frac{\sigma}{2}+\frac{3}{2}}}\rd t\bigg) .
		\end{aligned}
	\end{equation*}
Performing change of variables in the second term on the right-hand side we get
$$
	{K}_\sigma(x,y) \leq \frac{C_\ell}{|x-y|^{1+\sigma}}.
$$This proves the first statement. For the second one, from \eqref{eq:Mehler} if $t\in (0,t_0)$ and $|x-y|\leq 1$, then we have
	\begin{equation*}
		\begin{aligned}
			M_t(x,y)
			&\geq 	\exp\bigg(\frac{x^2+y^2}{4}-\frac{(x^2+y^2)(1-e^{-t_0})}{4(1+e^{-t_0})}\bigg)\frac{1}{(1 - e^{-2t})^{1/2}}
			\exp\left(
			- \frac{e^{-t}(x-y)^2}{2(1 - e^{-2t})}
			\right)
			\\
			&= 	\exp\bigg(\frac{x^2+y^2}{4} \Big(1-\frac{ 1-e^{-t_0}}{ 1+e^{-t_0}}\Big)\bigg)\frac{1}{(1 - e^{-2t})^{1/2}}
			\exp\left(
			- \frac{e^{-t}(x-y)^2}{2(1 - e^{-2t})}
			\right)
			\\
			&\geq C_{t_0} 	\exp\bigg(\frac{x^2+y^2}{4} \Big(1-\frac{ 1-e^{-t_0}}{ 1+e^{-t_0}}\Big)\bigg)\frac{1}{(2t)^{1/2}}
			\exp\left(
			- \frac{(x-y)^2}{4t}
			\right).
		\end{aligned}
	\end{equation*}
Note that, the constant $C_{t_0}$ does not depend on $x,y$ since $|x-y|\leq 1$.	Therefore, we obtain
	\begin{equation*}
		\begin{aligned}
			K_\sigma(x,y)&\geq  \int_0^{t_0} \frac{M_t(x,y)}{t^{\frac{\sigma}{2}+1}} \rd t \\
			&\geq  C_{t_0} 	\exp\bigg(\frac{x^2+y^2}{4} \Big(1-\frac{ 1-e^{-t_0}}{ 1+e^{-t_0}}\Big)\bigg) \int_0^{t_0}\frac{1}{t^{\sigma/2+3/2}}
			\exp\left(
			- \frac{(x-y)^2}{4t}\right)\rd t.
		\end{aligned}
	\end{equation*}
Performing a change of variables and using $|x-y|\leq 1$ we get
\begin{equation*}
	\begin{aligned}
\int_0^{t_0}\frac{1}{t^{\sigma/2+3/2}}
\exp\left(
- \frac{(x-y)^2}{4t}\right)\rd t &= \frac{2^{\sigma+1}}{|x-y|^{1+\sigma}}\int_{\frac{|x-y|^2}{4t_0}}^{\infty}e^{-\xi} \xi^{\frac{\sigma}{2}-\frac{1}{2}}\rd\xi
		\\
		&\geq \frac{2^{\sigma+1}}{|x-y|^{1+\sigma}}\int_{\frac{1}{4t_0}}^{\infty}e^{-\xi} \xi^{\frac{\sigma}{2}-\frac{1}{2}}\rd\xi.	 
	\end{aligned}
\end{equation*}
It follows that
	\begin{equation*}
	\begin{aligned}
		K_\sigma(x,y) 	&\geq  C_{t_0} 	\exp\bigg(\frac{x^2+y^2}{4} \Big(1-\frac{ 1-e^{-t_0}}{ 1+e^{-t_0}}\Big)\bigg) \frac{1}{|x-y|^{1+\sigma}}.
	\end{aligned}
\end{equation*}
	The proof is completed.
	\hfill
\end{proof}

By Lemmas \ref{lem:continuous} and \ref{lem:help} we obtain the following. 
\begin{lemma}\label{lem:embedding}
Let $1\leq p<\infty$ and $s>0$. If $s\not \in \NN$, then  
$W_{p}^s(\mathbb{R}^d,\gamma)$ is continuously embedded into  $W^s_{p,G}(\RRd,\gamma)$. Consequently, $f\in W^s_{p}(\RRd,\gamma)$ with either $s>\frac{1}{p}$ or $s=p=1$ is continuous on $\RRd$.
\end{lemma}

We now turn to the definition of the Hermite spaces.  The $d$-variate Hermite
polynomial $H_\bk$ for   $\bk\in \NNd_0$ is defined by
\begin{equation*}\label{H_bk}
	H_\bk(\bx) :=\prod_{j=1}^d H_{k_j}(x_j),
	\;\; \bx\in \RRd.
\end{equation*}
It is well known that the Hermite polynomials $(H_\bk)_{\bk \in \NNd_0}$ form an orthonormal basis of the Hilbert space $L_2(\RRd,\gamma)$ (see, e.g.,  \cite[Section 5.5]{Szego1939}). In particular,  every function $f \in L_2(\RRd,\gamma)$ admits a representation in terms of its Hermite series
\begin{equation}\label{H-series}
	f = \sum_{\bk \in \NNd_0} \widehat{f}(\bk) H_\bk \ \ {\rm with} \ \ \widehat{f}(\bk) := \int_{\RRd} f(\bx)\, H_\bk(\bx)\rd \gamma(\bx) 
\end{equation}
with convergence in the norm of $L_2(\RRd,\gamma)$ and there holds the Parseval's identity
\begin{equation*}\label{P-id}
	\norm{f}{L_2(\RRd,\gamma)}^2= \sum_{\bk \in \NNd_0} |\widehat{f}(\bk)|^2.
\end{equation*}
\begin{definition}
Let $s > 0$. The space $\mathcal{H}^s(\mathbb{R}^d,\gamma)$ consists of all functions $f \in L_2(\mathbb{R}^d,\gamma)$ that admit a representation in terms of the Hermite series \eqref{H-series} and for which the norm
\begin{equation*}
	\|f\|_{\mathcal{H}^s(\mathbb{R}^d,\gamma)} := \Bigg(\sum_{\bk \in \mathbb{N}_0^d} \bigg(\prod_{j=1}^d (k_j + 1)^s \bigg)  |\widehat{f}(\bk)|^2 \Bigg)^{1/2}<\infty
\end{equation*}
\end{definition}
It has been proved in \cite{DILP18} that for $s \in \NN$ we have
	\begin{equation}\label{N-eq}
		\norm{f}{\Wa}^2 \asymp \sum_{\bk \in \NNd_0} \prod_{j=1}^d \brac{k_j + 1}^s|\widehat{f}(\bk)|^2, \quad    f \in \Wa.
	\end{equation}
Hence, for $s \in \mathbb{N}$, the spaces $W^s_2(\mathbb{R}^d,\gamma)$ and $\mathcal{H}^s(\mathbb{R}^d,\gamma)$ can be identified in the sense of equivalent norms. This identification is extended to the case $s \notin \mathbb{N}$ in the following lemma.
%%%%%%%%%%%%%%%%%%%%%%%%%%%%

\begin{lemma}\label{lem:embedding2}
Let $s>0$. Then $$	W_{2}^s(\mathbb{R}^d,\gamma)=	\mathcal{H}^s(\mathbb{R}^d,\gamma)$$ in the sense of equivalent norms.
\end{lemma}
To prove the above lemma, we need the representation of derivatives of $f\in W^{s}_2(\RRd,\gamma)$, $s\in \NN$, in terms of Hermite polynomials  which was proved in \cite{DILP18}.
\begin{lemma}\label{lem:dick}
	Let $s\in \NN$. If $f\in W^{s}_2(\RRd,\gamma)$ and $\balpha\in\mathbb{N}_0^{d}$ such that $|\balpha|_\infty\leq s$, then we have
	$$
	D^{\balpha}f(\bx)=\sum_{\bk\ge\balpha}\widehat{f}\left(\bk\right)\sqrt{\frac{\bk!}{\left(\bk-\balpha\right)!}}H_{\bk-\balpha}\left(\bx\right).
	$$
\end{lemma}
\begin{proof}[Proof of Lemma \ref{lem:embedding2}] We prove the case $s\not \in\NN$.
For simplicity and clarity, we consider the case $d=2$. The case $d>2$ can be treated analogously. We use the following formula for $v\in \mathcal{H}^\sigma(\RR,\gamma)$, $\sigma\in (0,1)$,
\begin{equation}\label{eq:help}
\begin{aligned}
	\int_{\mathbb{R}} \rd\gamma(x)
	\int_{\mathbb{R}}
	|v(x) - v(y)|^2 \, K_{2\sigma}(x,y)\, \rd\gamma(y)&=
	2\Gamma(-\sigma)\int_{\mathbb{R}} v(x)\, (-L_\gamma)^\sigma v(x) \, \rd\gamma(x)
\end{aligned}	 
\end{equation}
which can be found in \cite[page 6]{CCMP22}. For $\balpha\in \NN_0^2$ with $|\balpha|_\infty=\bar{s}$, denote $u=D^{\balpha}f$. By $A^\balpha_{1}$, $A^\balpha_{2}$ and $A^\balpha_{1,2}$ we denote three integral terms in \eqref{eq:quasi-norm2} corresponding to the three cases  $e=\{1\}$, $e=\{2\}$ and $e=\{1,2\}$. By using \eqref{eq:help} we have
\begin{equation*}
	\begin{aligned}
A^\balpha_{1}&=\int_{\RR^{2}}\rd\gamma(\bx)\int_{\RR}
	\left|u(x_1,x_2)-u(y_1,x_2)    \right|^2K_{2\tilde{s}}(x_1,y_1)  \rd\gamma(y_1)
	\\
	&=2\Gamma(-\tilde{s})\int_{\RR}\rd\gamma(x_2)\int_{\RR}u(\bx)(-L_\gamma)_1^{\tilde{s}} u(\bx) \, \rd\gamma(x_1)\,.
	\end{aligned}
\end{equation*}
Here $\big(L_\gamma\big)_j^{\tilde{s}}$ is the fractional Ornstein–Uhlenbeck operator acting on the $j$-th coordinate. Using \eqref{H-series} and \eqref{eq:eigen} we get
\begin{equation*}
	\begin{aligned}
A_1^\balpha	&=2\Gamma(-\tilde{s})\sum_{\bk\in\NN_0^2}\widehat{u}(\bk)\int_{\RR^2}u(\bx)(-L_\gamma)_1^{\tilde{s}}  H_{\bk}(\bx) \, \rd\gamma(\bx)
	\\
	&=2\Gamma(-\tilde{s})\sum_{\bk\in\NN_0^2}k_1^{\tilde{s}}\widehat{u}(\bk)\int_{\RR^2}u(\bx) H_{\bk}(\bx) \, \rd\gamma(\bx)
		=2\Gamma(-\tilde{s})\sum_{\bk\in\NN_0^2}k_1^{\tilde{s}}|\widehat{u}(\bk)|^2\,.
	\end{aligned}
\end{equation*}
The case $e=\{2\}$ is carried similarly. Now we consider the case $e=\{1,2\}$. Since
$$\Delta^{\{1,2\}}_{\by}u(\bx)=\Delta_{y_2,2} \Delta_{y_1,1}u(\bx)$$
by \eqref{eq:help} we have
\begin{equation*}
	\begin{aligned}
	A^{\balpha}_{1,2}	&=\int_{\RR^{2}}\rd\gamma(\bx)\int_{\RR^2}
\big[\Delta^{\{1,2\}}_{\by}u(\bx)\big]^2K_{2\tilde{s}}(x_1,y_1)K_{2\tilde{s}}(x_2,y_2)   \rd\gamma(\by)
\\		
&=	2\Gamma(-\tilde{s})\int_{\RR^{2}}\rd\gamma(\bx)\int_{\RR}
	 \Delta_{y_1,1} u(\bx)\big(-L_\gamma\big)_2^{\tilde{s}} \Delta_{y_1,1} u(\bx) K_{2\tilde{s}}(x_1,y_1)  \rd\gamma(y_1)\,.
	\end{aligned}
\end{equation*}
We write
$$
\Delta_{y_1,1} u(\bx)=u(x_1,x_2)-u(y_1,x_2)=\sum_{\bk\in \NN_0^2} \widehat{u}(\bk) \Delta_{y_1} h_{k_1}(x_1)h_{k_2}(x_2)
$$
and using \eqref{eq:eigen} to obtain
\begin{equation*}
	\begin{aligned}
A^{\balpha}_{1,2}
&= 	2\Gamma(-\tilde{s})
\sum_{\bk,\bell\in\NN_0^2}k_2^{\tilde{s}} \widehat{u}(\bk)\widehat{u}(\bell)\int_{\RR^{2}}\rd\gamma(\bx)
\\
&\ \ \ \ \times \int_{\RR}\big[\Delta_{y_1} h_{k_1}(x_1)\Delta_{y_1} h_{\ell_1}(x_1)\big]h_{k_2}(x_2)h_{\ell_2}(x_2) K_{2\tilde{s}}(x_1,y_1)  \rd\gamma(y_1)
\\
&=	2\Gamma(-\tilde{s})
\sum_{\bk,\bell\in\NN_0^2}k_2^{\tilde{s}} \widehat{u}(\bk)\widehat{u}(\bell)\delta_{k_2,\ell_2}\int_{\RR}\rd\gamma(x_1)\int_{\RR}\big[\Delta_{y_1} h_{k_1}(x_1) \Delta_{y_1} h_{\ell_1}(x_1)\big] K_{2\tilde{s}}(x_1,y_1)  \rd\gamma(y_1).
	\end{aligned}
\end{equation*}
Here $\delta_{k_2,\ell_2}$ is the Kronecker delta. Denote the the integral factor by $I_{k_1,\ell_1}$. If $k_1\not =\ell_1$ then 
\begin{equation*}
	\begin{aligned}
I_{k_1,\ell_1}&=\int_{\RR}\rd\gamma(x_1)\int_{\RR}\big[h_{k_1}(x_1)-h_{k_1}(y_1)\big]  h_{\ell_1}(x_1) K_{2\tilde{s}}(x_1,y_1)  \rd\gamma(y_1)
		\\
& -\int_{\RR}\rd\gamma(x_1)\int_{\RR}\big[h_{k_1}(x_1)-h_{k_1}(y_1)\big] h_{\ell_1}(y_1) K_{2\tilde{s}}(x_1,y_1)  \rd\gamma(y_1)
		\\
		&=2\int_{\RR}\rd\gamma(x_1)\int_{\RR}\big[h_{k_1}(x_1)-h_{k_1}(y_1)\big]  h_{\ell_1}(x_1) K_{2\tilde{s}}(x_1,y_1)  \rd\gamma(y_1).
	\end{aligned}
\end{equation*}
Here, in the second equality, we interchange the roles of $x_1$ and $y_1$. By \eqref{eq:OU-operator} and \eqref{eq:eigen} we obtain
\begin{equation*}
	\begin{aligned}
		I_{k_1,\ell_1}
		&=2\Gamma(-\tilde{s})\int_{\RR} h_{\ell_1}(x_1)(-L_\gamma)^{\tilde{s}} h_{k_1}(x_1)\rd\gamma(x_1)	=0.
	\end{aligned}
\end{equation*}
In the case $k_1 =\ell_1$ we have by \eqref{eq:help} and \eqref{eq:eigen}
\begin{equation*}
	\begin{aligned}
I_{k_1,\ell_1}	&=\int_{\RR}\rd\gamma(x_1)	\int_{\RR}\big[\Delta_{y_1} h_{k_1}(x_1)\big]^2 K_{2\tilde{s}}(x_1,y_1)  \rd\gamma(y_1)
		\\
		&=2\Gamma(-\tilde{s})\int_{\RR} h_{k_1}(x_1) (-L_\gamma)^{\tilde{s}} h_{k_1}(x_1)  \rd\gamma(x_1)=2\Gamma(-\tilde{s}) k_1^{\tilde{s}}.
	\end{aligned}
\end{equation*}
Hence
\begin{equation}\label{eq:equivalent-01}
	\begin{aligned}
A^{\balpha}_{1,2}=4\Gamma(-\tilde{s})^2
\sum_{\bk \in\NN_0^2} (k_1k_2)^{\tilde{s}}|\widehat{u}(\bk)|^2.
\end{aligned}
\end{equation}
From Lemma \ref{lem:dick} we have
\begin{equation*}
\begin{aligned}
\widehat{u}(\bk) =\int_{\RR^2} D^{\balpha}f(\bx) H_\bk(\bx)\rd\gamma(\bx) &=\sum_{\bell\ge\balpha}\widehat{f}\left(\bell\right)\sqrt{\frac{\bell!}{\left(\bell-\balpha\right)!}}\int_{\RR^2}H_{\bell-\balpha}\left(\bx\right) H_\bk(\bx)\rd\gamma(\bx)
\\
&= \widehat{f}(\balpha+\bk)\sqrt{\frac{(\balpha+\bk)!}{\bk!}}\,.
\end{aligned}
\end{equation*}
Inserting this into \eqref{eq:equivalent-01} we finally obtain
\begin{equation*}
\begin{aligned}
A^\balpha_{1}+A^\balpha_{2}+A^\balpha_{1,2}&\asymp\sum_{\bk\in\NN_0^2}\big(k_1^{\tilde{s}}+k_2^{\tilde{s}}+(k_1k_2)^{\tilde{s}}\big)  \frac{(\balpha+\bk)!}{\bk!}	|\widehat{f}(\balpha+\bk)|^2
\end{aligned}
\end{equation*}
which together with \eqref{N-eq} implies
\begin{equation*}
	\begin{aligned}
\|f\|_{W^{\bar{s}}_p(\mathbb{R}^d,\gamma)}^2+\sum_{|\balpha|_\infty = \bar{s}}	\Big(A^\balpha_{1}+A^\balpha_{2}+A^\balpha_{1,2} \Big)
		\asymp \sum_{\bk\in \NN_0^2} (1+k_1)^{s}(1+k_2)^s |\widehat{f}(\bk)|^2.		 
	\end{aligned}
\end{equation*}
The proof is completed.\hfill
\end{proof}
\section{Numerical integration for fractional Gaussian Sobolev spaces}\label{sec:numer}
In this section, we derive a quadrature rule on $\mathbb{R}^d$ for the numerical integration of functions in fractional Gaussian Sobolev spaces by appropriately assembling quadrature rules defined on the cube $\IId := \big[-\frac{1}{2}, \frac{1}{2}\big]^d$ for functions from classical Sobolev spaces with mixed smoothness. We show that this construction preserves the convergence rate. As a consequence, we establish the asymptotic order of $\Int_n(\boldsymbol{W}^s_{p}(\RRd,\gamma))$
and $\Int_n(\boldsymbol{W}^s_{p,G}(\RRd,\gamma))$.

We begin by recalling known results on the numerical integration of function spaces with dominating mixed smoothness on $\IId$. The problem of approximating integrals of functions from such spaces has attracted considerable attention over the past six decades; see   \cite{Ko59,Hl62,Bakhvalov1963,Fro76,Tem86b,Tem1990,Tem1991,Tem03,Du93,Du97,Tr10b,Hil10,HMOU14,NoWo10,KrNo15,UU2015,DP2010,DU15,Ma13}. Historical remarks and additional references can be found in the recent survey \cite[Sect.\ 8.8--8.9]{DTU18B}.
The following theorem is well known.
\begin{theorem}\label{thm:help}
	Let $1< p<\infty$ and $s>\frac{1}{p}$ and $s\not \in \NN$. Then we have
	\begin{equation}\label{Int_n>}
		\Int_n(\mathring{\boldsymbol{W}^s_p}(\IId)) \asymp \Int_n({\boldsymbol{W}}^s_p(\IId))\asymp n^{-s}(\log n)^{(d-1)(1-\frac{1}{p})},
	\end{equation}
where $\mathring{W}^s_p(\IId)$ denotes the space of functions in $\Wap(\RRd)$ whose support is the subset of $\IId$. 
\end{theorem}
The lower bound in \eqref{Int_n>} is derived  by constructing a fooling function in $\BWpgamma$ whose support does not contain any integration nodes, and then applying the corresponding inequality 
\begin{equation*} 
	\Int_n\big(\mathring{\boldsymbol{W}^s_p}(\IId)\big) \ge \inf_{{\bx_1,\dots,\bx_n} \subset \IId } \ \sup_{\substack{f\in \mathring{\boldsymbol{W}^s_p}(\IId) , f(\bx_i)= 0,\ i =1,\dots,n}}\left|\int_{\IId} f(\bx)\rd\bx\right|\,.
\end{equation*}
The upper bound in \eqref{Int_n>} for $\Int_n(\mathring{\boldsymbol{W}^s_p}(\IId))$ is obtained by means of Frolov-type cubature rules. In the case $d = 2$, this upper bound can alternatively be achieved using Fibonacci lattice rules. The upper bound of $\Int_n({\boldsymbol{W}}^s_p(\IId))$ follows from a modified Frolov cubature formula by change of variables, which preserves the order of convergence up to a constant factor. See \cite{Du93,Du97,NUU17,Tem93B} and the references therein for further details on this transformation technique.

Combining Lemma~\ref{lem:estimate} (1) and Lemma~\ref{lem:embedding}, and noting that $|\rho(\bx)|\leq C$ for all $\bx\in \IId$ we conclude that $\mathring{W}^s_p(\IId)\subset W^s_{p}(\RRd,\gamma)\subset W^s_{p,G}(\RRd,\gamma)$ and if $f\in \mathring{W}^s_p(\IId)$ we have
$$
\|f\|_{W^s_{p,G}(\RRd,\gamma)} \ll  \|f\|_{W^s_{p}(\RRd,\gamma)} \ll  \|f\|_{\mathring{W}^s_p(\IId)}.
$$
Therefore
\begin{equation} \label{eq:lower0}
\Int_n\big({\boldsymbol{W}}^s_{p,G}(\mathbb{R}^d,\gamma)\big) \gg\Int_n\big({\boldsymbol{W}}^s_{p}(\mathbb{R}^d,\gamma) \big)
\gg	 \Int_n(\mathring{\boldsymbol{W}^s_p}(\IId))  \gg n^{-s}(\log n)^{(d-1)(1-\frac{1}{p})} .
\end{equation}	
%%%%%%%%%%%%%%%%%%%%%%%%%%%%%%%%%
We now proceed to prove the following theorems.
\begin{theorem}\label{thm:main1}
Let $1< p<\infty$ and $s>\frac{1}{p}$ and $s\not \in \NN$. Then we have
	 \begin{equation*}
	\Int_n\big({\boldsymbol{W}}^s_{p,G}(\mathbb{R}^d,\gamma)\big)\gg n^{-s}(\log n)^{(d-1)(1-\frac{1}{p})}
	 \end{equation*}
and if $2<p<\infty$
	 \begin{equation*}
	\Int_n\big({\boldsymbol{W}}^{s}_{p,G}(\mathbb{R}^d,\gamma)\big)\asymp n^{-s}(\log n)^{(d-1)(1-\frac{1}{p})}.
\end{equation*}
\end{theorem}

\begin{theorem}\label{thm:main2}
	Let $1<p<\infty$, $s>\frac{1}{p}$ and $s\not \in \NN$. Then have
	\begin{equation*}  \Int_n\big({\boldsymbol{W}}^s_{p}(\mathbb{R}^d,\gamma) \big)
		\asymp	n^{-s} (\log n)^{(d-1)(1-\frac{1}{p})}. 
	\end{equation*}
	Consequently, if $s\in \RR$ and $s>\frac{1}{2}$
	\begin{equation*} 
		\Int_n\big({\boldsymbol{\mathcal{H}}}^s(\mathbb{R}^d,\gamma)\big)\asymp	n^{-s} (\log n)^{\frac{d-1}{2}}.
	\end{equation*}
\end{theorem}
Since the proofs of these theorems are analogous, we provide a detailed proof of Theorem \ref{thm:main1} and outline the proof of Theorem \ref{thm:main2}.

\begin{proof}[Proof of Theorem \ref{thm:main1}] The lower bound has already been established in \eqref{eq:lower0}. To derive the upper bound, we prove the following result.
	Let $s > 0$ with $s \notin \mathbb{N}$, $2 < p < \infty$, and let $a > 0$, $b \ge 0$.
Assume that for the quadrature
\begin{equation}\label{I_m(f)}
	I_m(g): = \sum_{i=1}^m \lambda_{i,m} g(\bx_{i,m}), \ \ \{\bx_{1,m},\ldots,\bx_{m,m}\}\subset \IId,
\end{equation}
holds the convergence rate
\begin{equation}\label{IntError-a,b}
	\bigg|\int_{\IId} g(\bx) \rd \bx  - I_m(g)\bigg| \leq C m^{-a} (\log m)^b \|g\|_{\Wpmix}, 
	\ \  g\in \Wpmix.
\end{equation}
Then based on  $I_m$, we will construct  a quadrature  on $\RRd$  which approximates the integral $I^\gamma(f)$ with the same convergence rate for 
$f \in W^s_{p,G}(\RRd,\gamma)$, i.e.,
\begin{equation} 	\label{IntError}
	\bigg|\int_{\RRd}f(\bx) \rd\gamma(\bx) - I_n(f)\bigg| 
	\ll n^{-a}  (\log n)^b \|f\|_{W^s_{p,G}(\RRd,\gamma)}.
\end{equation}

To this end for $f\in W^s_{p,G}(\RRd,\gamma) $ we decompose the integral $I^\gamma(f)$ to the sum of component integrals over the  integer-shifted $d$-cubes $\IId_{\bk}$ by 
\begin{align} \label{Int_RRd}
I^\gamma(f)=	 \sum_{\bk \in \ZZd}\int_{\IId_\bk}f_\bk(\bx)\rho_\bk(\bx)\rd \bx,
\end{align}
where  for $\bk \in \ZZd$, $\IId_\bk:=\bk+\IId$,
$f_\bk $ and $\rho_\bk$ denote  the restriction of $f$ and $\rho$ to $\IId_\bk$.
Next we will prove that $f_\bk(\cdot+\bk)\rho_\bk(\cdot+\bk)$ belongs to $\Wpmix$. We have
\begin{equation*}\label{eq:b4}
\begin{split}
	\|f_\bk(\cdot+\bk)\|_{W_p^{\bar{s}}(\IId)}&=\Bigg(\sum_{|\balpha|_\infty \leq \bar{s}} \|D^\balpha f_\bk(\cdot+\bk)\|_{L_p(\IId)}^p\Bigg)^{1/p}
	\\
&=\Bigg(\sum_{|\balpha|_\infty \leq \bar{s}} \|D^\balpha f_\bk\|_{L_p(\IId_\bk)}^p\Bigg)^{1/p}
\\ 
&=\Bigg(\sum_{|\balpha|_\infty \leq \bar{s}} (2\pi)^{d/2} \int_{\IId_\bk}e^{\frac{|\bx|^2}{2}}|D^\balpha f_\bk(\bx)| ^p\rd\gamma(\bx) \Bigg)^{1/p}.
\end{split}
\end{equation*}
Choose $p'$ such that $2<p'<p$. When $\bx \in \IId_{\bk}$ we have 
$e^{\frac{|\bx|^2}{2}}\leq C e^{\frac{p|\bk|^2}{2p'}}$. Therefore,
\begin{equation}\label{eq:bar{s}}
	\|f_\bk(\cdot+\bk)\|_{W_p^{\bar{s}}(\IId)}
	\leq C e^{\frac{|\bk|^2}{2p'}}\|f\|_{W^s_{p,G}(\RRd,\gamma)}.
\end{equation}
Now we estimate the term $[f_\bk(\cdot+\bk)]_{\Wap(\IId)}$. For any $\balpha\in \NN_0^d$ such that $|\balpha|_\infty\leq \bar{s}$ and $e\subset [d]$, $e\not =\emptyset$, we have
\begin{equation}\label{eq:quasi-norm1}
	\begin{aligned}
	&	\int_{\II^{|e|}}\int_{\IId}
		\frac{
			\big| \Delta^e_{\by_e+\bk_e}D^\balpha f_\bk(\bx+\bk)  \big|^p
		}{
			\prod_{j\in e}	|x_j - y_j|^{1+\tilde{s}p}
		}
	\rd\bx	\rd\by_e
	\\
		&\leq   \int_{\II^{|e|}_{\bk_e}}\int_{\IId_\bk} e^{\frac{|\bx|^2+|\by_e|^2}{2}}
		\frac{
			\big| \Delta^e_{\by_e}D^\balpha f_\bk(\bx) \big|^p
		}{
			\prod_{j\in e}	|x_j - y_j|^{1+\tilde{s}p}
		}
	 \rd\gamma(\bx)	\rd\gamma(\by_e)
		\\
			&\leq C e^{\frac{p|\bk|^2}{p'}} \int_{\II^{|e|}_{\bk_e}}\int_{\IId_\bk}
	\frac{
		\big| \Delta^e_{\by_e}D^\balpha f_\bk(\bx) \big|^p
	}{
		\prod_{j\in e}	|x_j - y_j|^{1+\tilde{s}p}
	} \rd\gamma(\bx)	\rd\gamma(\by_e)
	\\
		&\leq C e^{\frac{p|\bk|^2}{p'}} \|f\|_{W^s_{p,G}(\RRd,\gamma)}^p.
	\end{aligned}
\end{equation}
Altogether we find that
\begin{equation}
	\label{eq:norm-fwid}	\|f_\bk(\cdot+\bk)\|_{\Wpmix}\leq C e^{\frac{|\bk|^2}{p'} }\|f\|_{W^s_{p,G}(\RRd,\gamma)}.
\end{equation}

Since
$
W^r_p(\IId) $ with $r:=\lceil s\rceil$ is continuously embedded into $\Wap(\IId)=S^s_{p,p}B(\IId)$ we have 
$$\|\rho_{\bk}(\cdot+\bk)\|_{W^s_p(\IId)}\leq C\|\rho_{\bk}(\cdot+\bk)\|_{W^r_p(\IId)}=\Bigg(\sum_{|\balpha|_\infty \leq r} \|D^\balpha \rho\|_{L_p(\IId_\bk)}^p\Bigg)^{1/p}.$$
By \eqref{eq:hermite} we have
$$
\frac{\rd^k}{\rd x^k}\rho(x)=(-1)^k\sqrt{k!} H_k(x)\rho(x).
$$
Therefore for $\bx \in \IId_\bk$ we get
$$
|D^\balpha \rho(\bx)|=\Big|(-1)^{|\balpha|_1}\sqrt{\balpha!}H_\balpha(\bx) \rho(\bx)\Big| \leq Ce^{-\frac{|\bx|^2}{2\tau}} \leq   Ce^{-\frac{|\bk|^2}{2\tau'}}
	$$
for some $\tau'$ and $\tau$ such that $1<\tau<\tau'<\frac{p'}{2}<\infty $. Here we need the condition $2<p'<\infty$. This implies that
\begin{align} \label{g_bk}
	\|\rho_{\bk}(\cdot+\bk)\|_{\Wpmix}  \leq  Ce^{-\frac{|\bk|^2}{2\tau'}}
\end{align}
with $C$ independent of $\bk \in \ZZd$. Since $\Wpmix$ is a multiplication algebra (see \cite[Theorem 3.16]{NgS17}), from \eqref{eq:norm-fwid} and \eqref{g_bk}  we have that 
\begin{align*} \label{multipl-algebra1}
f_{\bk}(\cdot+\bk)\rho_{\bk}(\cdot+\bk)\in \Wpmix,
\end{align*}
 and 
 \begin{equation*}  
\begin{aligned} \label{multipl-algebra2}
	\|f_{\bk}(\cdot+\bk)\rho_{\bk}(\cdot+\bk)\|_{\Wpmix} 
	& 
	\leq C \|f_{\bk}(\cdot+\bk)\|_{\Wpmix}  \cdot \|\rho_{\bk}(\cdot+\bk)\|_{\Wpmix}
	\\
	& \leq C e^{\frac{|\bk|^2}{p'}-\frac{|\bk|^2}{2\tau'}}\|f\|_{W^s_{p,G}(\RRd,\gamma)}.
\end{aligned}
\end{equation*}
We choose $\delta>0$ so that
\begin{equation}  \label{eq:delta}
	\max \bigg\{e^{-\frac{p|\bk|^2}{2p'}\big(1-\frac{1}{p}\big)},
	e^{\frac{|\bk|^2}{p'}-\frac{|\bk|^2}{2\tau'}}\bigg\}
	\leq 
	C e^{-\delta |\bk|^2}
\end{equation}
for $\bk\in \ZZd$, and therefore,
\begin{align}  \label{f_{bk}<}
\|f_{\bk}(\cdot+\bk)\rho_{\bk}(\cdot+\bk)\|_{\Wpmix} 
	& \leq C e^{-\delta |\bk|^2}\|f\|_{W^s_{p,G}(\RRd,\gamma)}, \qquad  \bk\in \ZZd.
\end{align}

For a given $n \in \NN$, we define $n_\bk\in \NN_0$, $\bk\in \ZZ^d$, such that $
\sum_{\bk \in \ZZd}  n_\bk  \le n
$. Then we take ``shifted" quadratures $I_{n_\bk}$  of the form \eqref{I_m(f)} to approximate the integral $\int_{\IId_\bk}f_\bk(\bx)\rho_\bk(\bx)\rd \bx$  in  \eqref{Int_RRd}. 
In the next step, we assemble these shifted integration nodes to construct a quadrature rule on $\mathbb{R}^d$ for approximating $I^\gamma(f)$. We now describe this construction in detail.

  We define	for $n\in \NN$,
\begin{equation} \label{xi-int}	
	\xi_n =  \sqrt{\delta^{-1} 2 a(\log n)}\,,
\end{equation}
and for $\bk \in \ZZd$ with $|\bk|< \xi_n$
\begin{equation*} \label{n_bk}
	n_{\bk}= 
	\lfloor \varrho n  e^{-\frac{\delta}{2 a}|\bk|^2}  \rfloor,
\end{equation*}
where $\delta$ is in \eqref{eq:delta} and \begin{equation}\label{eq-c2-2.23}
	\varrho:=\Bigg(\sum_{j=0}^{\infty}\left[\left(2j+1\right)^{d}-\left(2j-1\right)^{d}\right]e^{-\frac{\delta}{2a}j^{2}}\Bigg)^{-1}.
\end{equation} We have 
$$
\sum_{\left|\bk\right|<\xi_{n}} n_{\bk} \leq \sum_{\left|\bk\right|<\xi_{n}} \varrho n e^{-\frac{\delta}{2 a}\left|\bk\right|^{2}}\leq \varrho n \sum_{\left|\bk\right|_{\infty}<\xi_{n}} e^{-\frac{\delta}{2 a}\left|\bk\right|_{\infty}^{2}}\leq \varrho n \sum_{j=0}^{\left\lfloor\xi_{n}\right\rfloor}\sum_{\left|\bk\right|_{\infty}=j}e^{-\frac{\delta}{2 a}j^2}.
$$
Since
$$|\{\bk \in \mathbb{Z}^{d}: \left|\bk\right|_{\infty}=j\}|=\left(2j+1\right)^{d}-\left(2j-1\right)^{d},$$ from
\eqref{eq-c2-2.23} we get
\begin{equation}\label{eq:<n}
\sum_{\left|\bk\right|<\xi_{n}} n_{\bk} \leq \varrho n \sum_{j=0}^{\left\lfloor\xi_{n}\right\rfloor}\left[\left(2j+1\right)^{d}-\left(2j-1\right)^{d}\right]e^{-\frac{\delta}{2 a}j^{2}} \leq n.
\end{equation}

Since $f_{\bk}(\cdot+\bk)\rho_{\bk}(\cdot+\bk)\in \Wpmix$ by \eqref{I_m(f)} we can define quadrature
$$
I_{n_\bk}\big[f_{\bk}(\cdot+\bk)\rho_{\bk}(\cdot+\bk)\big]:=\sum_{j=1}^{n_\bk} \lambda_{j,n_\bk}  f_{\bk}\big(\bx_{j,n_\bk}+\bk\big)\rho_{\bk}\big(\bx_{j,n_\bk}+\bk\big)
$$
which together with \eqref{IntError-a,b} and \eqref{f_{bk}<} derive  the estimate
\begin{align*}
	\bigg|\int_{\IId_\bk} f_{\bk}(\bx)\rho_{\bk}(\bx)\rd \bx 
	& - I_{n_\bk}(f_{\bk}(\cdot+\bk)\rho_{\bk}(\cdot+\bk))\bigg| 
	\\& = \bigg|\int_{\IId} f_\bk(\bx+\bk)\rho_{\bk}(\bx+\bk)\rd \bx- I_{n_\bk}\big[f_{\bk}(\cdot+\bk)\rho_{\bk}(\cdot+\bk)\big]\bigg|
	\\
	& \leq  Cn_{\bk}^{-a} (\log n_{\bk})^b e^{-\delta |\bk|^2} \|f\|_{W^s_{p,G}(\RRd,\gamma)}
	\\
	& \leq C ( n e^{-\frac{\delta}{2a}|\bk|^2} )^{-a} (\log n)^b
	e^{-\delta |\bk|^2}\|f\|_{W^s_{p,G}(\RRd,\gamma)}
	\\
	&= C   n^{-a} (\log n)^be^{-\frac{|\bk|^2\delta}{2}}\|f\|_{W^s_{p,G}(\RRd,\gamma)}.
\end{align*}
Hence,
\begin{equation}\label{eq:first-term}
\begin{aligned}
	\sum_{|\bk|< \xi_n} \bigg|\int_{\IId_\bk} f_{\bk}(\bx)\rho_{\bk}(\bx)\rd \bx 
	&- I_{n_\bk}\big[f_{\bk}(\cdot+\bk)\rho_{\bk}(\cdot+\bk)\big]\bigg| 
	\\
	& \leq C \sum_{|\bk|< \xi_n} n^{-a}  (\log n)^be^{-\frac{|\bk|^2\delta}{2}}\|f\|_{W^s_{p,G}(\RRd,\gamma)}
	\\&
	\leq C n^{-a}  (\log n)^b\|f\|_{W^s_{p,G}(\RRd,\gamma)}.
\end{aligned}	 
\end{equation}

 We define the quadrature for the approximate integration of  $f$ by
\begin{equation} \label{I_n^gamma}
	\begin{aligned}
I_n(f)&:=\sum_{|\bk|< \xi_n}I_{n_\bk}\big[f_{\bk}(\cdot+\bk)\rho_{\bk}(\cdot+\bk)\big]
\\
&
= \sum_{|\bk|< \xi_n}\sum_{j=1}^{n_\bk} \lambda_{j,n_\bk}  \rho\big(\bx_{j,n_\bk}+\bk\big)f_{\bk}\big(\bx_{j,n_\bk}+\bk\big).		 
	\end{aligned}
\end{equation}
The  integration nodes of this quadrature  are
\begin{equation*} \label{int-nodes}
\big\{\bx_{j,n_\bk}+\bk: |\bk|< \xi_n, \, j=1,\ldots,n_\bk\big\}
\subset \RRd,
\end{equation*}
and the  integration weights 
$$
\big\{\lambda_{j,n_\bk}  \rho\big(\bx_{j,n_\bk}+\bk\big): |\bk|< \xi_n, \, j=1,\ldots,n_\bk\big\}.
$$
Due to \eqref{eq:<n}, the number of integration nodes does not exceed $n$. By construction, all integration nodes lie within a ball of radius $\sqrt{d}/2 + \xi_n$. Moreover, the density of the integration nodes decays exponentially in $|\bk|$ as one moves from the origin in $\mathbb{R}^d$ toward the boundary of this ball. Consequently, the set of integration nodes is highly sparse.

From \eqref{Int_RRd} and \eqref{I_n^gamma}
it follows that
\begin{equation}\label{eq:decomposition}
	\begin{aligned}
	\bigg|\int_{\RRd}f(\bx)\rd \gamma(\bx) - I_n(f)\bigg| 
	&\leq \sum_{|\bk|< \xi_n} \bigg|\int_{\IId_\bk} f_{\bk}(\bx)\rho_{\bk}(\bx)\rd \bx - 
	I_{n_\bk}\big[f_{\bk}(\cdot+\bk)\rho_{\bk}(\cdot+\bk)\big]\bigg| 
	\\ &
	+ \sum_{|\bk|\geq \xi_n}\bigg|\int_{\IId_\bk}f_\bk(\bx)\rho_\bk(\bx)\rd \bx\bigg|.
\end{aligned}	 
\end{equation}
For each term in the second sum, by H\"older's inequality and \eqref{eq:delta} we obtain
	\begin{align*}
		\bigg|\int_{\IId_\bk} f_\bk(\bx) \rho_{\bk}(\bx)\rd \bx \bigg| 
		&\leq \bigg(\int_{\IId_\bk}|f_\bk(\bx)|^p \rho_\bk(\bx)\rd \bx \bigg)^{\frac{1}{p}} \bigg(\int_{\IId_\bk}\rho_\bk(\bx) \rd \bx \bigg)^{1-\frac{1}{p}} 
		\\
		&\leq C e^{-\frac{p|\bk|^2}{2p'}(1-\frac{1}{p})}\|f\|_{W^s_{p,G}(\RRd,\gamma)}
		\\
		&
		\leq Ce^{-\delta |\bk|^2}\|f\|_{W^s_{p,G}(\RRd,\gamma)}.
	\end{align*}
Therefore
\begin{equation*} 
	\begin{aligned}
		\sum_{|\bk|\geq \xi_n}\bigg|\int_{\IId_\bk}f_\bk(\bx)\rho_\bk(\bx)\rd \bx\bigg|
		&
		\leq C\sum_{\ell=\lceil\xi_n^2\rceil}^\infty   \sum_{|\bk|^2=\ell}e^{-\ell \delta}\|f\|_{W^s_{p,G}(\RRd,\gamma)}
		\\
		&  \leq C\sum_{\ell=\lceil\xi_n^2\rceil}^\infty   \sum_{|\bk|_\infty\leq \sqrt{\ell}}e^{-\ell \delta}\|f\|_{W^s_{p,G}(\RRd,\gamma)}
		\\
		&  \leq C \sum_{\ell=\lceil\xi_n^2\rceil}^\infty   e^{-\ell \delta}\ell^{d/2} \|f\|_{W^s_{p,G}(\RRd,\gamma)}.
	\end{aligned}
\end{equation*} 
Fix $\varepsilon \in (0,1/2)$ we get
\begin{equation*}\label{eq-epsilon01}
	\begin{aligned}
		\sum_{|\bk|\geq \xi_n}\bigg|\int_{\IId_\bk}f_\bk(\bx)\rho_\bk(\bx)\rd \bx\bigg|
&
	\leq  C e^{-\xi_n^2\delta(1-\varepsilon)}\sum_{\ell=\lceil\xi_n^2\rceil}^\infty   e^{-\ell\varepsilon \delta}\ell^{d/2}\|f\|_{W^s_{p,G}(\RRd,\gamma)}
		\\
	&
\leq C e^{-\xi_n^2\delta(1-\varepsilon)}\|f\|_{W^s_{p,G}(\RRd,\gamma)}.
	\end{aligned}
\end{equation*} 
Using \eqref{xi-int} we arrive at
\begin{equation}\label{eq:second-term}
	\begin{aligned}
		\sum_{|\bk|\geq \xi_n}\bigg|\int_{\IId_\bk}f_\bk(\bx)\rho_\bk(\bx)\rd \bx\bigg| 	&  \leq  Ce^{-2a(1-\varepsilon)\log n}\|f\|_{W^s_{p,G}(\RRd,\gamma)}	\\
		&	
		\leq C n^{-a}  (\log n)^b\|f\|_{W^s_{p,G}(\RRd,\gamma)}.
\end{aligned}
\end{equation}
From \eqref{eq:first-term}, \eqref{eq:decomposition}, and \eqref{eq:second-term} statement  \eqref{IntError} is proved. Now the upper bound follows by Theorem \ref{thm:help}. 
	\hfill
\end{proof}

\begin{proof}[Proof of Theorem \ref{thm:main2}] The lower bound has already been established in \eqref{eq:lower0}. Proof of the upper bound follows similarly to that of Theorem \ref{thm:main1}. We explain why the result can be extended to the case $1 < p < \infty$. Let $1<\tau<\tau'<p'<p<\infty$. Arguing as in \eqref{eq:bar{s}} we obtain
	\begin{equation*} \label{eq:estimate1}
		\|f_\bk(\cdot+\bk)\|_{W_p^{\bar{s}}(\IId)}
		\leq  Ce^{\frac{|\bk|^2}{2p'}}\|f\|_{\Wpgamma}.
	\end{equation*}
Using Lemma \ref{lem:estimate}, with $\balpha\in \NN_0^d$ such that $|\balpha|_\infty\leq \bar{s}$ and $e\subset \{1,\ldots,d\}$, $e\not =\emptyset$, as in  \eqref{eq:quasi-norm1} we get
\begin{equation*}
	\begin{aligned}
		&\int_{\II^{|e|}}\int_{\IId}
		\frac{
			\big| \Delta^e_{\by_e+\bk_e}D^\balpha f_\bk(\bx+\bk)  \big|^p
		}{
			\prod_{j\in e}	|x_j - y_j|^{1+\tilde{s}p}
		}
		\rd\bx	\rd\by_e
		\\
		&\leq  C_{t_0}\int_{\II^{|e|}_{\bk_e}}\int_{\IId_\bk}
			\big| \Delta^e_{\by_e}D^\balpha f_\bk(\bx) \big|^p e^{-\frac{|\bx|^2+|\by_e|^2}{4} \big(1-\frac{ 1-e^{-t_0}}{ 1+e^{-t_0}}\big)}\bigg(\prod_{i\in e}K_{p\tilde{s}}(x_i,y_i)\bigg)
		\rd\bx	\rd\by_e
		\\
	&\leq  C_{t_0}\int_{\II^{|e|}_{\bk_e}}\int_{\IId_\bk} 
\big| \Delta^e_{\by_e}D^\balpha f_\bk(\bx) \big|^p e^{\frac{|\bx|^2+|\by_e|^2}{4} \big(1+\frac{ 1-e^{-t_0}}{ 1+e^{-t_0}}\big)}\bigg(\prod_{i\in e}K_{p\tilde{s}}(x_i,y_i)\bigg)
\rd\gamma(\bx)	\rd\gamma(\by_e)
		\\
		&\leq   C_{t_0}  e^{\frac{p|\bk|^2}{2p'} \big(1+\frac{ 1-e^{-t_0}}{ 1+e^{-t_0}}\big)} \|f\|_{\Wpgamma}^p.
	\end{aligned}
\end{equation*}
Therefore we find that
$$
\|f_\bk(\cdot+\bk)\|_{\Wpmix}\leq  C e^{\frac{|\bk|^2}{2p'} \big(1+\frac{ 1-e^{-t_0}}{ 1+e^{-t_0}}\big)} \|f\|_{\Wpgamma}
$$	
which implies
 \begin{equation*}  
	\begin{aligned} 
		\|f_{\bk}(\cdot+\bk)\rho_{\bk}(\cdot+\bk)\|_{\Wpmix} 
		& 
		\leq C \|f_{\bk}(\cdot+\bk)\|_{\Wpmix}  \cdot \|\rho_{\bk}(\cdot+\bk)\|_{\Wpmix}
		\\
		& \leq C e^{\frac{|\bk|^2}{2p'} \big(1+\frac{ 1-e^{-t_0}}{ 1+e^{-t_0}}\big)-\frac{|\bk|^2}{2\tau'}}\|f\|_{\Wpgamma}.
	\end{aligned}
\end{equation*}	
Now we choose $t_0$ such that $$\frac{1}{2p'}\bigg(1+\frac{ 1-e^{-t_0}}{ 1+e^{-t_0}}\bigg)<\frac{1}{2\tau'}
$$
and $\delta>0$ so that
\begin{equation*}   
	\max \bigg\{e^{-\frac{p|\bk|^2}{2p'}\big(1-\frac{1}{p}\big)},
	e^{\frac{|\bk|^2}{2p'} \big(1+\frac{ 1-e^{-t_0}}{ 1+e^{-t_0}}\big)-\frac{|\bk|^2}{2\tau'}}\bigg\}
	\leq 
	C e^{-\delta |\bk|^2}.
\end{equation*}
Next step is the same as in the proof of Theorem \ref{thm:main1}. The second statement is a consequence of the first one, Lemma \ref{lem:embedding2} and \eqref{DK}, \eqref{N-eq}.  
The proof is completed.\hfill
\end{proof}
\section{Numerical experiments}\label{Sec-4}

In this section we run numerical experiment for the cases $d=1$ and $s\in \{{1.1, 1.8,2.7}\}$. We consider the algorithm for the Hermite spaces $\mathcal{H}^s(\RR,\gamma)={W}^s_2(\RR,\gamma)$. 
Notice that for $s >1$, $\mathcal{H}^s(\RR,\gamma)$ is a separable reproducing kernel Hilbert space with the reproducing kernel
\begin{equation}\label{RK}
	K_s(x,y)= \sum_{k \in \NN_0} \frac{1}{(k+1)^s} H_k(x)H_k(y).
\end{equation}
Assume that we have the quadrature 
\begin{equation*}  
	I_n(f) := \sum_{i=1}^n \lambda_{i}  f(x_i),\ \ n\in \NN,
\end{equation*}
to approximate the integral of $f\in \mathcal{H}^s(\RR,\gamma)$. Then the error of this quadrature  is given by
\begin{equation}\label{eq:err}
	 	\begin{aligned}
{{\rm err}}&:=\sup_{f\in \boldsymbol{\mathcal{H}}^s(\RR,\gamma)}|I^\gamma(f)-I_n(f)|= \int_{\mathbb{R}} \int_{\mathbb{R}} K_s\left(x, y\right) \rho(x) \rho(y) \mathrm{d} x \mathrm{d} y \\
	 	&\ \quad -2 \sum_{i=1}^{n} \lambda_{i} \int_{\mathbb{R}} K_s(x, x_{i}) \rho(x) \mathrm{d} x+\sum_{i, j=1}^{n} \lambda_{i} \lambda_{j} K_s\left(x_{i}, x_{j}\right),
\end{aligned}
\end{equation}
see, e.g., \cite[Section 4]{DILP18}. By using the representation \eqref{RK} and a straightforward computation, we obtain
$$
{{\rm err}}=\Bigg(\bigg(1-\sum_{i=1}^n \lambda_i\bigg)^2+\sum_{k=1}^\infty \frac{1}{(1+k)^s}\bigg(\sum_{i=1}^n\lambda_iH_k(x_i)\bigg)^2\Bigg)^{1/2},
$$

For the numerical integration of functions in $\mathring{W}^s_2(\II)$, we employ equidistant nodes with equal weights.  We then apply  the method of change of variable
\begin{equation*}
	W^s_2(\II) \to \mathring{W}^s_2(\II): \ f\mapsto \psi' f(\psi),
\end{equation*} where 
\begin{equation*}\label{psin}
	\psi(x) = \begin{cases}
		C\int_{0}^x  (\frac{1}{4}-\xi^2)^4\,\rd\xi&\text{if} \ t\in [-\frac{1}{2},\frac{1}{2}],\\ 
		\frac{1}{2},&\text{if}\ t>\frac{1}{2},\\ 
		-\frac{1}{2} ,&\text{if}\ t<-\frac{1}{2}\,,
\end{cases}
\end{equation*}
and $C=\big(\int_{-1/2}^{1/2} (\frac{1}{4}-\xi^2)^4 \,\rd\xi\big)^{-1}$ to get asymptotically optimal integration nodes and weights for functions in $W^s_2(\II)$.

To form the quadrature $I_n(f)$ for $f\in \mathcal{H}^s(\RR,\gamma)$ we use assemble method described in Section \ref{sec:numer}. 
For the numerical computation the error \eqref{eq:err} is replaced by the  truncated version
\begin{equation}\label{eq:errorm}
{{\rm err}}_m=\Bigg(\bigg(1-\sum_{i=1}^n \lambda_i\bigg)^2+\sum_{k=1}^{m} \frac{1}{(k+1)^s}\bigg(\sum_{i=1}^n\lambda_iH_k(x_i)\bigg)^2\Bigg)^{1/2}.
\end{equation}

In our experiments, we set $m = 3 \times 10^4$. The results, presented in Figure~\ref{Fig-1}, show that the worst-case integration error for functions in the Hermite spaces $\mathcal{H}^s(\RR,\gamma)$ with $s \in \{1.1; 1.7; 2.8\}$ exhibits a convergence rate of $\mathcal{O}(n^{-s})$. These numerical findings are in agreement with the theoretical results established in this paper.

In this section, we also present numerical experiments illustrating the behavior of the error estimate \eqref{eq:errorm} for $\tfrac{1}{2} < s \leq 1$. Figure~\ref{Fig-2} displays the convergence of ${\rm err}_m$ for $s \in {0.6, 0.8, 1}$. It is noteworthy that the convergence rate $O(n^{-s})$ of ${\rm err}_m$ is already clearly observed even for relatively small numbers of sampling points. This observation is in good agreement with the theoretical findings of Theorem~\ref{thm:main2} for Hermite spaces.
\begin{figure}
	\includegraphics[height=9.2cm]{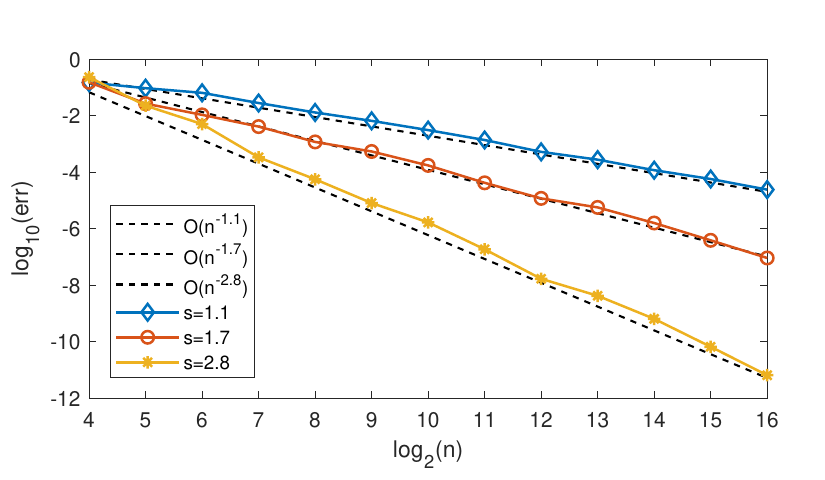}
	\caption{Worst-case error of numerical integration for functions in Hermite spaces}
	\label{Fig-1}
\end{figure}

\begin{figure}
	\includegraphics[height=9.2cm]{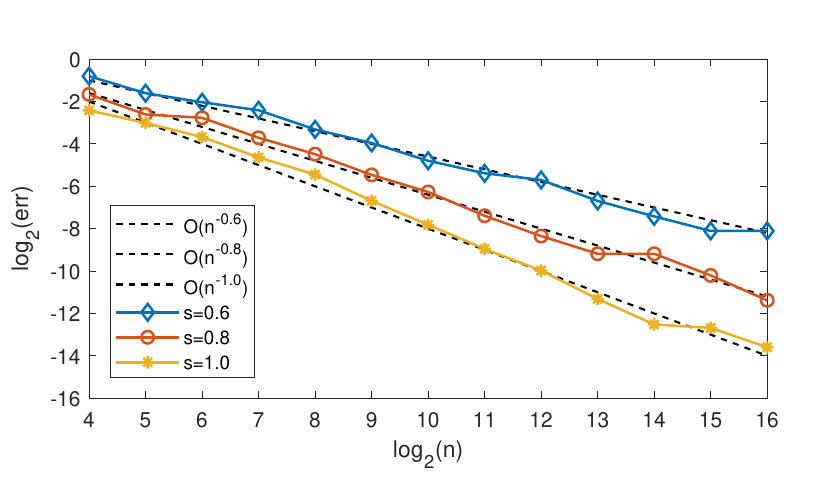}
	\caption{Worst-case error of numerical integration for functions in Hermite spaces with smoothness less than $1$}
	\label{Fig-2}
\end{figure}
\noindent
{\bf Acknowledgments:} A part of this work was done when  the author was working at the Vietnam Institute for Advanced Study in Mathematics (VIASM). He would like to thank  the VIASM for providing a fruitful research environment and working condition. 
\bibliographystyle{abbrv}
\bibliography{AllBib}

\begin{thebibliography}{10}

\bibitem{Bakhvalov1963}
N.~S. Bakhvalov.
\newblock Optimal convergence bounds for quadrature processes and integration
  methods of monte carlo type for classes of functions.
\newblock {\em Zhurnal Vychislitel'noi Matematiki i Matematicheskoi Fiziki},
  4:5--63, 1963.

\bibitem{CCMP22}
A.~Carbotti, S.~Cito, D.~A.~L. Manna, and D.~Pallara.
\newblock {Gamma convergence of Gaussian fractional perimeter}.
\newblock {\em Adv. Calc. Var.}, 16:571--595, 2023.

\bibitem{DN23}
D.~{D\~ung} and V.~K. Nguyen.
\newblock {Optimal numerical integration and approximation of functions on
  $\mathbb{R}^d$ equipped with Gaussian measure}.
\newblock {\em IMA J. Numer. Anal.}, 44:1242--–1267, 2024.

\bibitem{DTU18B}
D.~{D\~ung}, V.~N. Temlyakov, and T.~Ullrich.
\newblock {\em {Hyperbolic Cross Approximation}}.
\newblock Advanced Courses in Mathematics - CRM Barcelona,
  Birkh\"auser/Springer, 2018.

\bibitem{DILP18}
J.~Dick, C.~Irrgeher, G.~Leobacher, and F.~Pillichshammer.
\newblock {On the optimal order of integration in Hermite spaces with finite
  smoothness}.
\newblock {\em SIAM J. Numer. Anal.}, 56:684--707, 2018.

\bibitem{DP2010}
J.~Dick and F.~Pillichshammer.
\newblock {\em Digital Nets and Sequences: Discrepancy Theory and Quasi-Monte
  Carlo Integration}.
\newblock Cambridge University Press, Cambridge, 2010.

\bibitem{Du93}
V.~Dubinin.
\newblock Cubature formulas for classes of functions with bounded mixed
  difference.
\newblock {\em Sbornik: Mathematics}, 76:283--292, 1993.

\bibitem{Du97}
V.~Dubinin.
\newblock Cubature formulae for {B}esov classes.
\newblock {\em Izvestiya: Mathematics}, 61:259--283, 1997.

\bibitem{DU15}
D.~Dung and T.~Ullrich.
\newblock Lower bounds for the integration error for multivariate functions
  with mixed smoothness and optimal fibonacci cubature for functions on the
  square.
\newblock {\em Math. Nachr.}, 288:743--762, 2015.

\bibitem{Fro76}
K.~K. Frolov.
\newblock Upper error bounds for quadrature formulas on function classes.
\newblock {\em Doklady Akademii Nauk SSSR}, 231:818--821, 1976.

\bibitem{HV09}
M.~Hansen and J.~Vybiral.
\newblock The jawerth--franke embedding of spaces of dominating mixed
  smoothness.
\newblock {\em Georgian Math. J.}, 16(4):667--682, 2009.

\bibitem{Hil10}
A.~Hinrichs.
\newblock Discrepancy of {H}ammersley points in {B}esov spaces of dominating
  mixed smoothness.
\newblock {\em Math. Nachr.}, 283:478--488, 2010.

\bibitem{HMOU14}
A.~Hinrichs, L.~Markhasin, J.~Oettershagen, and T.~Ullrich.
\newblock Optimal quasi-monte carlo rules on order 2 digital nets for the
  numerical integration of multivariate periodic functions.
\newblock {\em Numer. Math.}, 134:163--196, 2014.

\bibitem{Hl62}
E.~Hlawka.
\newblock {Zur angen\"aherten {B}erechnung mehrfacher {I}ntegrale}.
\newblock {\em Monatsh. Math.}, 66:140--151, 1962.

\bibitem{IKLP2015}
C.~Irrgeher, P.~Kritzer, G.~Leobacher, and F.~Pillichshammer.
\newblock {Integration in Hermite spaces of analytic functions}.
\newblock {\em J. Complexity}, 31:308--404, 2015.

\bibitem{IL2015}
C.~Irrgeher and G.~Leobacher.
\newblock {High-dimensional integration on the $\mathbb{R}^d$, weighted
  Hermitespaces, and orthogonal transforms}.
\newblock {\em J. Complexity}, 31:174--205, 2015.

\bibitem{Ko59}
N.~M. Korobov.
\newblock {Approximate evaluation of repeated integrals}.
\newblock {\em Dokl. Akad. Nauk SSSR}, 124, 1959.

\bibitem{KrNo15}
D.~Krieg and E.~Novak.
\newblock A universal algorithm for multivariate integration.
\newblock {\em Found. Comput. Math.}, 17:895--916, 2016.

\bibitem{LMP20}
A.~Lunardi, G.~Metafune, and D.~Pallara.
\newblock {The Ornstein-Uhlenbeck semigroup in finite dimensions}.
\newblock {\em Philos. Trans. R. Soc. Lond. Ser. A Math. Phys. Eng. Sci.}, 378,
  2020.

\bibitem{Ma13}
L.~Markhasin.
\newblock Discrepancy and integration in function spaces with dominating mixed
  smoothness.
\newblock {\em Diss. Math.}, 494:1--81, 2013.

\bibitem{NgS17}
V.~K. Nguyen and W.~Sickel.
\newblock {Pointwise multipliers for Sobolev and Besov spaces of dominating
  mixed smoothness}.
\newblock {\em J. Math. Anal. Appl.}, 452:62--90, 2017.

\bibitem{NUU17}
V.~K. Nguyen, M.~Ullrich, and T.~Ullrich.
\newblock Change of variable in spaces of mixed smoothness and numerical
  integration of multivariate functions on the unit cube.
\newblock {\em Constr. Approx.}, 46:69--108, 2017.

\bibitem{NoWo10}
E.~Novak and H.~Wo{\'z}niakowski.
\newblock {\em {Tractability of Multivariate Problems, Volume II: Standard
  Information for Functionals}}.
\newblock EMS Tracts in Mathematics, Vol. 12, Eur. Math. Soc. Publ. House,
  Z\"urich, 2010.

\bibitem{ST87B}
H.~Schmeisser and H.~Triebel.
\newblock {\em {Topics in Fourier Analysis and Function Spaces}}.
\newblock Chichester; New York : Wiley, 1987.

\bibitem{Szego1939}
G.~Szeg\"o.
\newblock {\em {Orthogonal Polynomials}}.
\newblock Amer. Math. Soc., 1939.

\bibitem{Tem86b}
V.~N. Temlyakov.
\newblock {Approximation of functions with bounded mixed derivative}.
\newblock {\em Trudy MIAN}, 178:1--112, 1986.

\bibitem{Tem1990}
V.~N. Temlyakov.
\newblock {On a way of obtaining lower estimates for the errors of quadrature
  formulas}.
\newblock {\em Matem. Sb.}, pages 1403--1413, 1990.

\bibitem{Tem1991}
V.~N. Temlyakov.
\newblock {Error estimates for Fibonacci quadrature formulas for classes of
  functions with a bounded mixed derivative}.
\newblock {\em Trudy Mat. Inst. Steklov}, 200:327--335, 1991.

\bibitem{Tem93B}
V.~N. Temlyakov.
\newblock {\em {Approximation of Periodic Functions}}.
\newblock Computational Mathematics and Analysis Series, Nova Science
  Publishers, Inc., Commack, NY., 1993.

\bibitem{Tem03}
V.~N. Temlyakov.
\newblock {Cubature formulas, discrepancy, and nonlinear approximation}.
\newblock {\em J. Complexity}, 19, 2003.

\bibitem{Tr10b}
H.~Triebel.
\newblock {\em {Bases in Function Spaces, Sampling, Discrepancy, Numerical
  Integration}}.
\newblock EMS Tracts in Mathematics, European Mathematical Society (EMS),
  Z\"urich, 2010.

\bibitem{UU2015}
M.~{U}llrich and T.~{U}llrich.
\newblock {T}he role of {F}rolov's cubature formula for functions with bounded
  mixed derivative.
\newblock {\em SIAM J. Numer. Anal.}, 54, No. 2:969--993, 2016.

\end{thebibliography}

\end{document}